\title{Choquet extension of non-monotone submodular setfunctions}
\author{László Lovász\footnote{Alfréd Rényi Institute of Mathematics, Budapest, Hungary.
Research supported by ERC Synergy Grant No.~810115.}}
\date{}

\documentclass[12pt,fleqn]{article}
\usepackage{amsmath,amssymb,graphicx,bbm,bm,delarray,pict2e,ntheorem,stackengine}
\usepackage{hyperref}
\usepackage[latin2]{inputenc}
\usepackage[normalem]{ulem}
\usepackage{etoolbox}

\newtheorem{theorem}{Theorem}[section]

\newtheorem{lemma}[theorem]{Lemma}

\newtheorem{corollary}[theorem]{Corollary}

\theorembodyfont{\rmfamily}
\newtheorem{remark}[theorem]{Remark}
\newtheorem{example}[theorem]{Example}

\newtheorem{exc}[theorem]{Exercise}

\newenvironment{proof}{\medskip\noindent{\bf Proof. }}{\hfill$\square$\medskip}
\newenvironment{proof*}[1]{\medskip\noindent{\bf Proof of #1.}}{\hfill$\square$\medskip}
\AtEndEnvironment{proof}{\setcounter{claim}{0}}
\AtEndEnvironment{proof*}{\setcounter{claim}{0}}

\def\acup{{\cup\kern-0.8ex{}^\land}}
\def\bcap{{\cap\kern-0.8ex{}_\lor}}

\begin{document}

\def\AA{\mathcal{A}}\def\BB{\mathcal{B}}\def\CC{\mathcal{C}}
\def\DD{\mathcal{D}}\def\EE{\mathcal{E}}\def\FF{\mathcal{F}}
\def\GG{\mathcal{G}}\def\HH{\mathcal{H}}\def\II{\mathcal{I}}
\def\JJ{\mathcal{J}}\def\KK{\mathcal{K}}\def\LL{\mathcal{L}}
\def\MM{\mathcal{M}}\def\NN{\mathcal{N}}\def\OO{\mathcal{O}}
\def\PP{\mathcal{P}}\def\QQ{\mathcal{Q}}\def\RR{\mathcal{R}}
\def\SS{\mathcal{S}}\def\TT{\mathcal{T}}\def\UU{\mathcal{U}}
\def\VV{\mathcal{V}}\def\WW{\mathcal{W}}\def\XX{\mathcal{X}}
\def\YY{\mathcal{Y}}\def\ZZ{\mathcal{Z}}

\def\Ab{\mathbf{A}}\def\Bb{\mathbf{B}}\def\Cb{\mathbf{C}}
\def\Db{\mathbf{D}}\def\Eb{\mathbf{E}}\def\Fb{\mathbf{F}}
\def\Gb{\mathbf{G}}\def\Hb{\mathbf{H}}\def\Ib{\mathbf{I}}
\def\Jb{\mathbf{J}}\def\Kb{\mathbf{K}}\def\Lb{\mathbf{L}}
\def\Mb{\mathbf{M}}\def\Nb{\mathbf{N}}\def\Ob{\mathbf{O}}
\def\Pb{\mathbf{P}}\def\Qb{\mathbf{Q}}\def\Rb{\mathbf{R}}
\def\Sb{\mathbf{S}}\def\Tb{\mathbf{T}}\def\Ub{\mathbf{U}}
\def\Vb{\mathbf{V}}\def\Wb{\mathbf{W}}\def\Xb{\mathbf{X}}
\def\Yb{\mathbf{Y}}\def\Zb{\mathbf{Z}}

\def\ab{\mathbf{a}}\def\bb{\mathbf{b}}\def\cb{\mathbf{c}}
\def\db{\mathbf{d}}\def\eb{\mathbf{e}}\def\fb{\mathbf{f}}
\def\gb{\mathbf{g}}\def\hb{\mathbf{h}}\def\ib{\mathbf{i}}
\def\jb{\mathbf{j}}\def\kb{\mathbf{k}}\def\lb{\mathbf{l}}
\def\mb{\mathbf{m}}\def\nb{\mathbf{n}}\def\ob{\mathbf{o}}
\def\pb{\mathbf{p}}\def\qb{\mathbf{q}}\def\rb{\mathbf{r}}
\def\sb{\mathbf{s}}\def\tb{\mathbf{t}}\def\ub{\mathbf{u}}
\def\vb{\mathbf{v}}\def\wb{\mathbf{w}}\def\xb{\mathbf{x}}
\def\yb{\mathbf{y}}\def\zb{\mathbf{z}}

\def\Abb{\mathbb{A}}\def\Bbb{\mathbb{B}}\def\Cbb{\mathbb{C}}
\def\Dbb{\mathbb{D}}\def\Ebb{\mathbb{E}}\def\Fbb{\mathbb{F}}
\def\Gbb{\mathbb{G}}\def\Hbb{\mathbb{H}}\def\Ibb{\mathbb{I}}
\def\Jbb{\mathbb{J}}\def\Kbb{\mathbb{K}}\def\Lbb{\mathbb{L}}
\def\Mbb{\mathbb{M}}\def\Nbb{\mathbb{N}}\def\Obb{\mathbb{O}}
\def\Pbb{\mathbb{P}}\def\Qbb{\mathbb{Q}}\def\Rbb{\mathbb{R}}
\def\Sbb{\mathbb{S}}\def\Tbb{\mathbb{T}}\def\Ubb{\mathbb{U}}
\def\Vbb{\mathbb{V}}\def\Wbb{\mathbb{W}}\def\Xbb{\mathbb{X}}
\def\Ybb{\mathbb{Y}}\def\Zbb{\mathbb{Z}}

\def\Af{\mathfrak{A}}\def\Bf{\mathfrak{B}}\def\Cf{\mathfrak{C}}
\def\Df{\mathfrak{D}}\def\Ef{\mathfrak{E}}\def\Ff{\mathfrak{F}}
\def\Gf{\mathfrak{G}}\def\Hf{\mathfrak{H}}\def\If{\mathfrak{I}}
\def\Jf{\mathfrak{J}}\def\Kf{\mathfrak{K}}\def\Lf{\mathfrak{L}}
\def\Mf{\mathfrak{M}}\def\Nf{\mathfrak{N}}\def\Of{\mathfrak{O}}
\def\Pf{\mathfrak{P}}\def\Qf{\mathfrak{Q}}\def\Rf{\mathfrak{R}}
\def\Sf{\mathfrak{S}}\def\Tf{\mathfrak{T}}\def\Uf{\mathfrak{U}}
\def\Vf{\mathfrak{V}}\def\Wf{\mathfrak{W}}\def\Xf{\mathfrak{X}}
\def\Yf{\mathfrak{Y}}\def\Zf{\mathfrak{Z}}

\def\alphab{{\boldsymbol\alpha}}\def\betab{{\boldsymbol\beta}}
\def\gammab{{\boldsymbol\gamma}}\def\deltab{{\boldsymbol\delta}}
\def\etab{{\boldsymbol\eta}}\def\zetab{{\boldsymbol\zeta}}
\def\kappab{{\boldsymbol\kappa}}
\def\lambdab{{\boldsymbol\lambda}}\def\mub{{\boldsymbol\mu}}
\def\nub{{\boldsymbol\nu}}\def\pib{{\boldsymbol\pi}}
\def\rhob{{\boldsymbol\rho}}\def\sigmab{{\boldsymbol\sigma}}
\def\taub{{\boldsymbol\tau}}\def\epsb{{\boldsymbol\varepsilon}}
\def\epsilonb{{\boldsymbol\epsilon}} \def\inb{{\boldsymbol\in}}
\def\phib{{\boldsymbol\varphi}}\def\psib{{\boldsymbol\psi}}
\def\xib{{\boldsymbol\xi}}\def\omegab{{\boldsymbol\omega}}
\def\intl{\int\limits}
\def\sqprod{\mathbin{\square}}

\def\ybb{\mathbbm{y}}
\def\one{{\mathbbm1}}
\def\two{{\mathbbm2}}
\def\R{\Rbb}\def\Q{\Qbb}\def\Z{\Zbb}\def\N{\Nbb}\def\C{\Cbb}
\def\wh{\widehat}
\def\wt{\widetilde}
\def\var{\omega}
\def\eps{\varepsilon}
\def\sgn{{\rm sign}}
\def\dd{{\sf d}}
\def\Rv{\overleftarrow}
\def\Pr{{\sf P}}
\def\E{{\sf E}}
\def\T{^{\sf T}}
\def\proofend{\hfill$\square$}
\def\id{\hbox{\rm id}}
\def\conv{\hbox{\rm conc}}
\def\lin{\hbox{\rm lib}}
\def\conv{\hbox{\rm conc}}
\def\Dim{\hbox{\rm Dim}}
\def\const{\hbox{\rm const}}
\def\vol{\text{\rm vol}}
\def\diam{\text{\rm diam}}
\def\corank{\hbox{\rm cork}}
\def\cork{\hbox{\rm cork}}
\def\OR{\mathcal{OR}}
\def\GOR{\mathcal{GOB}}
\def\NOR{\mathcal{NOR}}
\def\LGOR{\mathcal{ALGOR}}
\def\cro{\text{\rm CR}}
\def\supp{\text{\rm sup}}
\def\grad{\text{\rm grad}}
\def\ro{\overline{\rho}}
\def\rk{\text{\rm rk}}
\def\srk{\hbox{\rm src}}
\def\diag{{\rm diag}}
\def\pw{{\sf w}_\text{\rm prod}}
\def\tw{{\sf w}_\text{\rm tree}}
\def\aw{{\sf w}_\text{\rm alb}}
\def\bw{{\sf bow}}
\def\ld{{\sf d}_{\rm loc}}
\def\hd{{\sf d}_{\rm har}}
\def\tv{\text{\rm tv}}
\def\Tr{\hbox{\rm Tar}}
\def\tr{\hbox{\rm tr}}
\def\Prob{\hbox{\rm Pr}}
\def\bl{\text{{\rm bl}}}
\def\abl{\hbox{{\rm abl}}}
\def\Id{\hbox{\rm Id}}
\def\aff{\text{\rm ra}}
\def\MC{\CC_{\max}}
\def\Inf{\text{\sf Inf}}
\def\Str{\text{\sf Str}}
\def\Rig{\text{\sf Rig}}
\def\Mat{\text{\sf Mat}}
\def\comm{{\sf comm}}
\def\maxcut{{\sf maxcut}}
\def\disc{\text{\sf disc}}
\def\cond{\Phi}
\def\val{\text{\sf val}}
\def\dist{d_{\rm qu}}
\def\dhaus{d_{\rm haus}}
\def\dlp{d_{\rm LP}}
\def\dact{d_{\rm act}}
\def\Ker{{\rm Ker}}
\def\Rng{{\rm Rng}}
\def\gdim{{\rm gdim}}
\def\gap{\text{\rm gap}}
\def\intl{\int\limits}
\def\et{\qquad\text{and}\qquad}
\def\fin{\text{\sf fin}}
\def\Bd{{\sf Bd}}
\def\ba{{\sf ba}}
\def\ca{{\sf ca}}
\def\matp{{\sf mc}}
\def\basp{{\sf bmc}}
\def\Find{{\sf Find}}
\def\Fbas{{\sf Fbas}}

\def\blim{{\sf blim}}
\def\sep{{\sf sep}}
\def\Hom{{\rm Hom}}
\def\bp{{\sf bp}}
\def\fg{\varphi}
\def\fgx{\varphi^{\lor\omega}}
\def\lc{{\rm lc}}
\def\li{{\rm li}}
\def\ui{{\rm ui}}
\def\ls{{\rm ls}}
\def\us{{\rm us}}
\def\Haus{\text{\rm Haus}}
\def\dbwedge{\,\doublebarwedge\,}
\def\scdot{\!\cdot\!}
\def\stimes{\!\times\!}
\def\upr{{\textstyle\prod_\omega}}

\long\def\ignore#1{}

\def\QR{{R^{cc}}}

\def\url{}

\maketitle



\addtolength{\baselineskip}{2pt}

\tableofcontents

\newpage

\begin{abstract}
In a seminal paper, Choquet introduced an integral formula to extend a monotone
increasing setfunction on a sigma-algebra to a (nonlinear) functional on
bounded measurable functions. The most important special case is when the
setfunction is submodular; then this functional is convex (and vice versa). In
the finite case, an analogous extension was introduced by this author; this is
a rather special case, but no monotonicity was assumed. In this note we show
that Choquet's integral formula can be applied to all submodular setfunctions,
and the resulting functional is still convex. We extend the construction to
submodular setfunctions defined on a set-algebra (rather than a sigma-algebra).
The main property of submodular setfunctions used in the proof is that they
have bounded variation. As a generalization of the convexity of the extension,
we show that (under smoothness conditions) a ``lopsided'' version of Fubini's
Theorem holds.
\end{abstract}

\section{Introduction}

Submodular setfunctions play an important role in potential theory, and perhaps
an even more important role in combinatorial optimization. The analytic line of
research goes back to the work of Choquet\cite{Choq}; the combinatorial, to the
work of Whitney \cite{Whit} and Rado \cite{Rado}, followed by the fundamental
work of Tutte \cite{Tut1,Tut2} and Edmonds \cite{Edm}. The two research lines
have not had much interaction so far.

Let $(J,\BB)$ be a sigma-algebra, and let $\Bd$ denote the Banach space of
bounded measurable functions on it, with the supremum norm. For an increasing
setfunction $\fg$ on $\BB$, Choquet introduced a non-linear functional,
extending $\fg$ from $0$-$1$ valued functions to all functions in $\Bd$. For
nonnegative functions $f\in\Bd$, this is defined by the integral
\begin{equation}\label{EQ:CHOQ-DEF0}
\wh\fg(f) = \intl_0^\infty \fg\{f\ge t\}\,dt,
\end{equation}
where $\{f\ge t\}$ is shorthand for $\{x\in J:~f(x)\ge t\}$. The functional
$\wh\fg$ is easily extended to all bounded measurable functions $f$ (see
below). The increasing property of $\fg$ is essential in this definition to
guarantee the integrability in \eqref{EQ:CHOQ-DEF0}.

One of the main classes of setfunctions to which this extension is applied in
Choquet's work consists of increasing submodular setfunctions (called
$2$-alternating by Choquet). A setfunction $\fg$ defined on a sigma-algebra
$(J,\BB)$ (say, all subsets of a finite set, or Borel subsets of $[0,1]$) is
{\it submodular}, if it satisfies the inequality
\begin{equation}\label{EQ:}
\fg(X\cup Y)+\fg(X\cap Y) \le \fg(X)+\fg(Y)\qquad(X,Y\in\BB).
\end{equation}
A setfunction satisfying this condition with equality for all $X$ and $Y$ is
called {\it modular}. A modular setfunction $\fg$ with $\fg(\emptyset)=0$ is
just a finitely additive signed measure, which we will call a {\it signed
charge}. A {\it charge} is a nonnegative signed charge. (See Rao and Rao
\cite{Rao} for the basics of the theory of charges.)

Perhaps the most important property of the functional $\wh\fg$, established by
Choquet, is that it is convex as a map $\Bd\to\R$ if $\fg$ is an increasing
submodular setfunction.

In the combinatorial world, an analogous extension of a setfunction $\fg$
defined on the subsets of a finite set was introduced in \cite{LL83}; this is a
rather special case, but no monotonicity was assumed. The convexity of the
extension was shown to be equivalent to the submodularity of $\fg$. Let us
point out that non-monotone submodular setfunctions play a central role in
combinatorial optimization; see \cite{Fuji} and \cite{Schr} for an in-depth
treatment of the subject and also of its history. Many of these applications
depend on the fact that the rank function of a matroid is submodular, but let
us point out that the cut-capacity function in the famous Max-Flow-Min-Cut
Theorem of Ford and Fulkerson \cite{FF} is a non-monotone submodular
setfunction. See Fujishige \cite{Fuji}, Schrijver \cite{Schr} and Frank
\cite{Frank11} for more.

Motivated by the goal of developing a limit theory of matroids, analogous to
the limit theory of graphs (see \cite{HomBook}), several aspects of the
interaction between the combinatorial and analytic theories of submodular
setfunctions have been formulated in \cite{Lov23}. A first crucial step is to
prove that Choquet's integral formula works for all (not necessarily
increasing) submodular setfunctions even in the analytic setting, and to show
that the resulting functional is still convex. The goal of this note is to
publish a proof of these facts. We define setfunctions with bounded variation,
show that the formula \eqref{EQ:CHOQ-DEF} works for them, and prove that every
bounded submodular setfunction has bounded variation. As a generalization of
the convexity property, we prove a ``lopsided'' version of Fubini's Theorem.
While not directly used it in this paper, but used in forthcoming applications,
we formulate most of our results in the framework of set-algebras (rather than
sigma-algebras).

\section{Choquet integrals}

\paragraph{Measurable functions on set-algebras.}
To extend our treatment to set-algebras instead of sigma-algebras leads to some
technical complications, which we need to discuss here.

If $\BB$ is a sigma-algebra, then a function $f:~J\to\R$ is $\BB$-measurable if
$\{f\ge t\}\in\BB$ for all $t\in\RR$. In the case of more general set families
$\BB$ like set-algebras, this definition would have some drawbacks; for
example, it would not imply that $-f$ is measurable. Even if we add the
condition that $\{f\le t\}\in\BB$, it would not follow that the sum of two
$\BB$-measurable functions is $\BB$-measurable. Hence we adopt the following
more general notion: we call $f$ {\it $\BB$-measurable}, if for every $s<t$
there is a set $A\in\BB$ such that $\{f\ge t\}\subseteq A\subseteq\{f\ge s\}$.
We denote by $\Bd=\Bd(\BB)$ the set of bounded $\BB$-measurable functions, and
set $\Bd_+=\{f\in\Bd:~f\ge 0\}$.

If $\{f\ge t\}\in\BB$ for all $t\in\R$, then $f$ is $\BB$-measurable, but not
the other way around. If $(J,\BB)$ is a sigma-algebra, then this notion
coincides with the traditional definition of measurability. It is known
(\cite{Rao}, Proposition 4.7.2) that for a set-algebra $(J,\BB)$, $\Bd$ is a
linear space. With the norm $\|f\|=\sup_{x\in J} |f(x)|$, the space $\Bd$ is a
Banach space.

\begin{remark}\label{REM:F-CONT}
The closely related notion of {\it $\FF$-continuous functions} was introduced
by Rao and Rao \cite{Rao}, Section 4.7. For a set-algebra $(J,\FF)$ and a
bounded function $f$, this is equivalent to $\FF$-measurability.
\end{remark}

\paragraph{Increasing setfunctions on sigma-algebras.}
We recall the definition of the integral of a bounded measurable function with
respect to an increasing setfunction on a sigma-algebra (Choquet \cite{Choq};
see also Denneberg \cite{Denn} and \v{S}ipo\v{s} \cite{Sipos}). Our main goal
in later sections is to show that for submodular setfunctions, the monotonicity
condition can be dropped, and instead of sigma-algebras, we can consider
set-algebras.

Let $(J,\BB)$ be a sigma-algebra and let $\Bd$ be the Banach space of bounded
measurable functions on $J$, with the supremum norm $\|.\|$. The ``Layer Cake
Representation'' of $f\in\Bd_+$ is the following elementary formula:
\[
f(x) = \intl_0^\infty \one_{f\ge t}(x)\,dt.
\]
Let $\fg$ be an increasing setfunction with $\fg(\emptyset)=0$. The {\it
Choquet integral} of a function $f\in\Bd_+$, motivated by the Layer Cake
Representation, is defined by
\begin{equation}\label{EQ:CHOQ-DEF}
\wh\fg(f) = \intl_0^\infty \fg\{f\ge t\}\,dt.
\end{equation}
This integral is well defined, since $\fg\{f\ge t\}$ is a bounded monotone
decreasing function of $t$, and the integrand is zero for sufficiently large
$t$. In the theory of ``nonlinear integral'' this quantity is often denoted by
$\int f\,d\fg$, but we prefer Choquet's notation $\wh\fg$.

More generally, if $f\in\Bd$ may have negative values, then we select any
$c\ge\|f\|$, and define
\begin{equation}\label{EQ:WH-DEF}
\wh\fg(f) = \intl_{-c}^c \fg\{f\ge t\}\,dt - c\fg(J) = \wh\fg(f+c)-c\fg(J).
\end{equation}
It is easy to see that this value is independent of $c$ once $c\ge\|f\|$.

For $S\in\BB$, we have $\wh\fg(\one_S)=\fg(S)$. So $\wh\fg$ can be considered
as an extension of $\fg$ from $0$-$1$ valued measurable functions to all
bounded measurable functions.  We call $\wh\fg$ the {\it Choquet extension} of
$\fg$. It is easy to see that the extension map $\fg\mapsto\wh\fg$ is linear
and monotone in the sense that if $\fg\le\psi$ on $\BB$, then
$\wh\fg\le\wh\psi$ on $\Bd_+$.

\paragraph{Increasing setfunctions on set-algebras.}
Let $(J,\BB)$ be a set-algebra, $\fg$, an increasing setfunction on $\BB$, and
$f\in\Bd_+$. Formula \eqref{EQ:CHOQ-DEF} is not necessarily meaningful, as the
level sets $\{f\ge t\}$ may not belong to $\BB$. One remedy is to consider the
following extensions of $\fg$ to $2^J$:
\begin{align*}
\fg^\ui(X)=\inf_{Y\in\BB\atop Y\supseteq X}\fg(Y)\quad\text{and}\quad
\fg^\ls(X)&=\sup_{Y\in\BB\atop Y\subseteq X}\fg(Y).
\end{align*}
(where the superscipt ``$\ui$'' refers to ``upper-infimum'' etc.), and replace
$\fg$ by either $\fg^\ls$ or by $\fg^\ui$, which agree with $\fg$ on $\BB$ and
are defined everywhere. (We note that if $\fg$ is submodular then so is
$\fg^\ui$, but $\fg^\ls$ is not submodular in general.)

\begin{lemma}\label{LEM:SET-ALG-INT}
Let $(J,\BB)$ be a set-algebra, let $\fg$ be an increasing setfunction on $\BB$
with $\fg(\emptyset)=0$, and let $f:~J\to\R$ be a bounded $\BB$-measurable
function. Then $\fg^\ui\{f\ge t\} = \fg^\ls\{f\ge t\}$ for almost all real
numbers $t$.
\end{lemma}

\begin{proof}
Trivially $\fg^\ls\le \fg^\ui$, and $\fg^\ls\le\fg\le \fg^\ui$ on $\BB$.
$\BB$-measurability of $\fg$ implies that for every $t$ and $\eps>0$ there is a
set $A\in\BB$ such that
\[
\{f\ge t\}\subseteq A \subseteq\{f\ge t-\eps\},
\]
which implies that
\[
\fg^\ui\{f\ge t\} \le \fg(A) \le \fg^\ls\{f\ge t-\eps\},
\]
and hence (assuming for simplicity that $f\ge 0$)
\begin{align*}
\intl_0^\infty \fg^\ui\{f\ge t\}\,dt &\le \intl_0^\infty \fg^\ls\{f\ge t-\eps\}\,dt
= \intl_{-\eps}^\infty \fg^\ls\{f\ge t\}\,dt\\
&= \eps\fg(J) + \intl_0^\infty \fg^\ls\{f\ge t\}\,dt.
\end{align*}
Letting $\eps\to 0$, we get that $\wh{\fg^\ui}=\wh{\fg^\ls}$, which implies the
lemma.
\end{proof}

We define $\wh\fg(f)=\wh{\fg^\ui}(f)=\wh{\fg^\ls}(f)$ for nonnegative functions
$f$. Formula $\wh\fg(f)= \wh\fg(f+c)-c\fg(J)$ ($c\ge\|f\|$) can be used to
define $\wh\fg(f)$ for all bounded functions.

\paragraph{Bounded variation.}\label{SSEC:BD-VAR}
The formula defining $\wh\fg(f)$ makes sense not only for increasing
setfunctions, but whenever $\fg\{f\ge t\}$ is an integrable function of $t$. A
necessary condition for this is that $\fg(\emptyset)=0$, which we are going to
assume. One sufficient condition for integrability is that $\fg\{f\ge t\}$ is
the difference of two bounded increasing functions of $t$. In turn, a
sufficient condition for this is that $\fg$ is the difference of two increasing
setfunctions. Analogously to the case of functions of single real variable,
this is equivalent to $\fg$ having {\it bounded variation} in the following
sense: there is a $K\in\R$ such that
\[
\sum_{i=1}^n |\fg(X_i)-\fg(X_{i-1})| \le K
\]
for every chain of subsets $\emptyset=X_0\subseteq X_1\subseteq\ldots\subseteq
X_n=J$. We denote the smallest $K$ for which this holds by $K(\fg)$. It is
clear that every setfunction with bounded variation is bounded, and every
increasing or decreasing setfunction on $\BB$ (with finite values) has bounded
variation. Every charge (finitely additive measure) has bounded variation.
Clearly setfunctions on $(J,\BB)$ with bounded variation form a linear
subspace.

\begin{lemma}\label{LEM:BD-VAR}
A setfunction on a set-algebra $(J,\BB)$ can be written as the difference of
two increasing setfunctions if and only if it has bounded variation.
\end{lemma}

For a signed charge, its positive and negative parts provide such a
decomposition.

\begin{proof}
If $\fg=\mu-\nu$, where $\mu$ and $\nu$ are increasing setfunctions, then $\mu$
and $\nu$ have bounded variation, and hence so does their difference.

Conversely, assume that $\fg$ has bounded variation. We may assume that
$\fg(\emptyset)=0$. Define $|a|_+=\max(0,a)$ and $|a|_-=\max(0, -a)$. For
$S\in\BB$, let
\begin{align}\label{EQ:ALPHA-BETA-DEF}
\mu(S)&=\sup \sum_{i=1}^n |\fg(X_i)-\fg(X_{i-1})|_+,\nonumber\\
\nu(S)&=\sup \sum_{i=1}^n |\fg(X_i)-\fg(X_{i-1})|_-,
\end{align}
where the suprema are taken over all chains $\emptyset=X_0\subset
X_1\subset\ldots\subset X_n=S$. Note that $\mu(S),\nu(S)\le K(\fg)$, and
\[
\sum_{i=1}^n |\fg(X_i)-\fg(X_{i-1})|_+ - \sum_{i=1}^n |\fg(X_i)-\fg(X_{i-1})|_-
= \sum_{i=1}^n (\fg(X_i)-\fg(X_{i-1})) = \fg(S),
\]
which implies that the suprema in \eqref{EQ:ALPHA-BETA-DEF} can be approximated
by the same chains of sets, and $\fg(S)=\mu(S)-\nu(S)$.

The setfunction $\mu$ is increasing. Indeed, let $S\subseteq T$; whenever a
sequence $\emptyset=X_0\subseteq X_1\subseteq\ldots\subseteq X_n=S$ competes in
the definition of $\mu(S)$, the sequence $\emptyset=X_0\subseteq
X_1\subseteq\ldots\subseteq X_n\subseteq X_{n+1}=T$ competes in the definition
of $\mu(T)$. Similarly $\nu$ is increasing.
\end{proof}

We call the pair $(\mu,\nu)$ above the {\it canonical decomposition} of $\fg$.
For a setfunction $\fg$ with bounded variation and $\fg(\emptyset)=0$, we
define the functional $\wh\fg:~\Bd\to\R$ by
\begin{equation}\label{EQ:WH-DEF2}
\wh\fg=\wh\mu-\wh\nu,
\end{equation}
where $\fg=\mu-\nu$ is the canonical decomposition of $\fg$. If $(J,\BB)$ is a
sigma-algebra, then formula \eqref{EQ:WH-DEF} makes sense and gives the same
value.

\paragraph{Simple properties.}
First, we consider the dependence of $\wh\fg(f)$ on $\fg$.

\begin{lemma}\label{LEM:HAT-LIN}
Let $(J,\BB)$ be a set-algebra. Then the map $\fg\mapsto\wh\fg$ is linear for
bounded setfunctions on $\BB$: if $c\in\R$, then $\wh{c\fg}=c\wh\fg$ and
$\wh{\fg+\psi}=\wh\fg+\wh\psi$.
\end{lemma}

\begin{proof}
The homogeneity is easy to check. For two increasing setfunctions $\phi$ and
$\psi$, it is easy to see that $(\phi+\psi)^\ui=\phi^\ui+\psi^\ui$, and hence
\[
\wh{\fg+\psi}=\wh{(\fg+\psi)^\ui} = \wh{\fg^\ui+\psi^\ui} = \wh{\fg^\ui}+\wh{\psi^\ui}=\wh\fg+\wh\psi.
\]

This implies that if we consider a general (non-increasing) setfunction $\fg$
with a decomposition $\fg=\mu_1-\nu_1$ into the difference of two increasing
setfunctions, then $\wh\fg=\wh\mu_1-\wh\nu_1$. Indeed, consider the canonical
decomposition $\fg=\mu-\nu$, then $\mu+\nu_1=\nu+\mu_1$, and hence
$\wh\fg=\wh\mu-\wh\nu=\wh\mu_1-\wh\nu_1$.

For two arbitrary setfunctions $\fg$ and $\psi$, let $\fg=\mu-\nu$ and
$\psi=\alpha-\beta$ be their canonical decompositions. Then $\fg+\psi =
(\mu+\alpha)-(\nu+\beta)$, and hence
\begin{align*}
\wh{\fg+\psi}=\wh{\mu+\alpha}-\wh{\nu+\beta}=\wh\mu+\wh\alpha-\wh\nu-\wh\beta=\wh\fg+\wh\psi.\tag*{$\square$}
\end{align*}
\end{proof}

Now we turn to the dependence on $f$. The functional $\wh{\fg}:~f\to\wh\fg(f)$
is, in general, not linear. It is trivially monotone increasing if $\fg$ is
increasing, but not for all setfunctions with bounded variation. It does have
some simple useful properties.

\begin{lemma}\label{LEM:HAT-CONT}
Let $\fg$ be a setfunction on a set-algebra $(J,\BB)$ with bounded variation
and with $\fg(\emptyset)=0$.

\smallskip

{\rm(a)} The functional $\wh\fg:~f\to\wh\fg(f)$ is positive homogeneous: if
$f\in\Bd$ and $c>0$, then $\wh\fg(cf)=c\wh\fg(f)$.

\smallskip

{\rm(b)} It satisfies the identities $\wh\fg(f+a)=\wh\fg(f)+a\fg(J)$ for every
real constant $a$ and $\wh\fg(-f)=-\wh{\fg^*}(f)$, where
$\fg^*(X)=\fg(J)-\fg(J\setminus X)$.

\smallskip

{\rm(c)} It has the Lipschitz property:
\[
|\wh\fg(f)-\wh\fg(g)|\le 2K(\fg)\|f-g\|
\]
for $f,g\in\Bd$.
\end{lemma}

\begin{proof}
Assertions (a) and (b) are straightforward to check for increasing
setfunctions, and follow in the general case by linearity (Lemma
\ref{LEM:HAT-LIN}). To verify (c), set $\eps=\|f-g\|$, and consider the
decomposition of $\fg$ into the difference of two increasing setfunctions
$\fg=\mu-\nu$, where we may assume that $\mu(\emptyset)=\nu(\emptyset)=0$, and
$\mu,\nu\le K(\fg)$. The inequalities $f\le g+\eps\le f+2\eps$ imply that
$\wh\mu(f)\le \wh\mu(g)+\eps\mu(J)\le\wh\mu(f)+2\eps\mu(J)$, and similarly
$\wh\nu(f)\le \wh\nu(g)+\eps\nu(J)\le\wh\nu(f)+2\eps\nu(J)$, which implies that
\[
|\wh\fg(f)-\wh\fg(g)|\le |\wh\mu(f)-\wh\mu(g)| + |\wh\nu(f)-\wh\nu(g)|\le \eps\mu(J)+\eps\nu(J)\le 2\eps K(\fg).
\]
\end{proof}

\paragraph{Submodularity and bounded variation.} The following is the key fact
allowing us to extend Choquet integration to non-monotone submodular
setfunctions.

\begin{theorem}\label{THM:SUBMOD-FINVAR}
Every bounded submodular setfunction on a set-algebra has bounded variation.
\end{theorem}

\begin{proof}
Consider a chain of subsets $\emptyset=X_0\subseteq X_1\subseteq\ldots\subseteq
X_n=J$ $(X_i\in\BB)$. Suppose that there is an index $1\le i\le n-1$ such that
$\fg(X_i)\le \fg(X_{i-1})$ and $\fg(X_i)\le \fg(X_{i+1})$. Let
$X_i'=X_{i-1}\cup(X_{i+1}\setminus X_i)$. Then $X_i\cap X_i'= X_{i-1}$ and
$X_i\cup X_i'= X_{i+1}$, so by submodularity,
\[
\fg(X_i)+\fg(X_i') \ge \fg(X_{i-1})+\fg(X_{i+1}).
\]
This implies that $\fg(X_i')\ge \fg(X_{i-1})$ and $\fg(X_i')\ge \fg(X_{i+1})$,
and also that
\begin{align*}
|\fg(X_{i+1})&-\fg(X_i)| + |\fg(X_i)-\fg(X_{i-1})|
= \fg(X_{i+1})+\fg(X_{i-1})-2\fg(X_i)\\
&\le \fg(X_{i+1})+\fg(X_{i-1})-2(\fg(X_{i-1})+\fg(X_{i+1}) -\fg(X'_i))\\
&= 2\fg(X_i')-\fg(X_{i+1})-\fg(X_{i-1})\\
&=|\fg(X_{i+1})-\fg(X_i')| + |\fg(X_i')-\fg(X_{i-1})|.
\end{align*}
So replacing $X_i$ by $X_i'$ does not decrease the sum $\sum_i
|\fg(X_i)-\fg(X_{i-1})|$. Repeating this exchange procedure a finite number of
times, we get a sequence $\emptyset=Y_0\subseteq Y_1\subseteq\ldots\subseteq
Y_n= J$ for which $\sum_i |\fg(Y_i)-\fg(Y_{i-1})|\ge\sum_i
|\fg(X_i)-\fg(X_{i-1})|$ and there is a $0\le j\le n$ such that
$\fg(Y_0)\le\ldots\le\fg(Y_j)\ge\fg(Y_{j+1})\ge\ldots\ge\fg(Y_n)$. For such a
sequence,
\[
\sum_{i=1}^n |\fg(Y_i)-\fg(Y_{i-1})| = 2\fg(Y_j)-\fg(J)-\fg(\emptyset).
\]
Since $\fg$ is bounded, this proves that $\fg$ has bounded variation.
\end{proof}

\begin{remark}\label{RE:DECOMPOSE}
By Lemma \ref{LEM:BD-VAR} and its proof, for every bounded submodular
setfunction a canonical decomposition is defined. This decomposition can be
more explicitly stated. Following the steps of the proofs, we get that
$\fg=\psi+(\fg-\psi)$, where $\psi=\fg^\ls|_\BB$. It is easy to see that $\psi$
is increasing and $\fg-\psi$ is decreasing. It would be very nice to find such
a decomposition where both terms are submodular. In the finite case, every
submodular setfunction can be written as the sum two submodular setfunctions,
one increasing and one decreasing; however, in the infinite case a
counterexample is given in \cite{BGILS}.
\end{remark}

\section{Convexity}\label{SEC:CONV}

We prove the key fact that every submodular setfunction $\fg$ extends to a
subadditive (and, since it is positive homogeneous, convex) functional
$\wh\fg$. For the increasing case, this was proved by Choquet \cite{Choq}; see
also \v{S}ipo\v{s} \cite{Sipos} and Denneberg \cite{Denn}, Chapter 6.

\begin{theorem}\label{THM:SUBMOD-UNCROSS2}
Let $\fg$ be a bounded submodular setfunction on a set-algebra $(J,\BB)$ with
$\fg(\emptyset)=0$, and let $f,g\in\Bd$. Then
\[
\wh\fg(f+g) \le  \wh\fg(f) + \wh\fg(g).
\]
\end{theorem}

Of course, the inequality generalizes to the sum of any finite number of
functions in $\Bd$. We give a proof using a basic combinatorial technique
called ``uncrossing'', illustrating the tight connection between the analytic
and combinatorial theories. To this end, we state and prove the following
special case first:

\begin{lemma}\label{LEM:SUBMOD-UNCROSS0}
Let $\fg$ be a bounded submodular setfunction on a set-algebra $(J,\BB)$ with
$\fg(\emptyset)=0$, let $H_1,\ldots,H_n\in\BB$, $a_1,\ldots a_n\in\Nbb$, and
$h=\sum_i a_i\one_{H_i}$. Then
\[
\wh\fg(h) \le \sum_{i=1}^n a_i \fg(H_i).
\]
If the sets $H_i$ form a chain, then equality holds for any setfunction $\fg$.
\end{lemma}

\begin{proof}
We may assume that $J$ is finite, since we may merge the atoms of the
set-algebra generated by $\HH$ to single points. The assertion about equality
is trivial, since then the sets $H_u$ are just the level sets of $h$. For the
general case, let $\HH$ be the multiset consisting of $a_i$ copies of $H_i$,
and let $|\HH|$ be the cardinality of $\HH$ as a multiset, i.e.,
$|\HH|=\sum_ia_i$. Then $h = \sum_{H\in \HH} \one_H$, and we want to prove that
\begin{equation}\label{EQ:FBF}
\wh\fg(h) \le \sum_{H\in \HH} \fg(H).
\end{equation}
Suppose that we find two sets $H_1,H_2\in\HH$ such that neither one of them
contains the other. Replace one copy of $H_1$ and of $H_2$ by $H_1'=H_1\cup
H_2$ and $H_2'=H_1\cap H_2$, and let $\HH'$ be the resulting multiset. Then
clearly $\sum_{H\in\HH'} \one_H = h$, and
\begin{equation}\label{EQ:H-GROW}
\sum_{H\in\HH'} \fg(H) \le \sum_{H\in\HH} \fg(H).
\end{equation}
by submodularity. Let us repeat this transformation as long as we can. Since we
stay with subsets of a finite set and $|\HH|$ does not change, but the quantity
$\sum_{H\in\HH}|H|^2$ increases at each step, the procedure must stop after a
finite number of iterations with a multiset that is a chain. As remarked above,
in this case equality holds, which proves the inequality in the lemma.
\end{proof}

\begin{proof*}{Theorem \ref{THM:SUBMOD-UNCROSS2}}
By \eqref{EQ:WH-DEF}, we may assume that $f$ and $g$ are nonnegative. If they
are integer-valued stepfunctions, then we express them by their layer cake
representation, and apply Lemma \ref{LEM:SUBMOD-UNCROSS0} to get the inequality
as stated. For rational-valued stepfunctions, the inequality follows by
scaling. The general case follows via approximation by stepfunctions and Lemma
\ref{LEM:HAT-CONT}(c).
\end{proof*}

\begin{corollary}\label{COR:SUBMOD-CONV}
Let $\fg$ be a setfunction with bounded variation on a set-algebra. Then
$\wh\fg$ is a convex functional if and only if $\fg$ is submodular.
\end{corollary}

\begin{proof}
The ``if'' part follows immediately by the homogeneity of $\wh\fg$ and Theorem
\ref{THM:SUBMOD-UNCROSS2}. To prove the ``only if'' part, suppose that $\wh\fg$
is convex. Since it is also positive homogeneous, we have
\[
\fg(S\cup T)+\fg(S\cap T) =
\wh\fg(\one_S+\one_T) \le \wh\fg(\one_S)+\wh\fg(\one_T) = \fg(S)+\fg(T).
\]
proving that $\fg$ is submodular.
\end{proof}

\section{Lopsided Fubini Theorem}\label{SEC:FUBINI}

Assuming that we are working in sigma-algebras (not merely set-algebras) and an
appropriate continuity of $\fg$, we can prove the following generalization of
Theorem \ref{THM:SUBMOD-UNCROSS2}. Let $\fg$ and $\psi$ be setfunctions defined
on the same set-algebra $(J,\BB)$. We say that a setfunction $\fg$ is {\it
uniformly continuous with respect to $\psi$}, if for every $\eps>0$ there is a
$\delta>0$ such that whenever $\psi(S\triangle T)<\delta$ ($S,T\in\BB$), then
$|\fg(S)-\fg(T)|<\eps$.

\begin{theorem}\label{THM:SUBMOD-UNCROSS3}
Let $(I,\AA,\lambda)$ and $(J,\BB,\pi)$ be probability spaces. Let $\fg\ge 0$
be a submodular setfunction on $(J,\BB)$
uniformly continuous with respect to $\pi$. Let $F:~ I\times J\to\R$ be a
bounded measurable function. Define $F_x(y)=F(x,y)$ and
\[
g(y)=\intl_I F(x,y)\,d\lambda(x).
\]
Then $\wh\fg(F_x)$ is an integrable function of $x$, and
\begin{equation}\label{EQ:MAIN1}
\wh\fg(g) \le \intl_I \wh\fg(F_x)\,d\lambda(x).
\end{equation}
\end{theorem}

Using the integral notation, we can write this inequality as
\[
\intl_J \intl_I F(x,y)\,d\lambda(x)\,d\fg(y) \le \intl_I \intl_J F(x,y)\,d\fg(y)\,d\lambda(x).
\]
So inequality \eqref{EQ:MAIN1} is a ``lopsided'' version of Fubini's Theorem.

\begin{proof}
By Lemma \ref{LEM:HAT-CONT}, we may assume that $0\le F\le 1$, and by scaling,
that $0\le\fg\le 1$. To prove that $\wh\fg(F_x)$ is an integrable function of
$x$, we note that it is bounded, and so it suffices to show the following
claim:

\medskip

\noindent{\bf Claim 1.} {\it $\wh\fg(F_x)$ is a measurable function of $x$ with
respect to the $\lambda$-completion $\overline{\AA}$ of $\AA$.}

\medskip

We start with proving this for indicator functions $F=\one_U$,
$U\in\AA\times\BB$. For $x\in I$, let $U_x=\{y\in J:~(x,y)\in U\}$. Claim 1 is
clearly true if $U$ is the union of a finite number product sets $S\times T$,
$S,T\in\AA$, since then $\wh\fg(F_x)$ is piecewise constant. For a general $U$,
measurability implies that there is a sequence of sets $W_n\subseteq I\times
J$, each a finite union of measurable product sets, such that $(\lambda\times
\pi)(U\triangle W_n)\to 0$ as $n\to\infty$. Hence $\pi(U_x\setminus (Y_n)_x)\to
0$ for $\lambda$-almost all $x\in I$. By the uniform continuity of $\fg$, this
implies that $\fg((Y_n)_x)\to \fg(U_x)$ for $\lambda$-almost all $x\in I$. So
$\fg(U_x)=\wh\fg(F_x)$ is an $\overline{\AA}$-measurable function of $x$.

This implies that Claim 1 also holds for any measurable stepfunction $F$.
Indeed, we can write $F=\sum_{i=1}^n a_i\one_{U_i}$ with some measurable sets
$U_1\subset U_2\subset\dots\subset U_n$ and coefficients $a_i>0$. Then
$F_x=\sum_{i=1}^n a_i\one_{(U_i)_x}$ and hence
\[
\wh\fg(F_x) =\sum_{i=1}^n a_i\fg((U_i)_x),
\]
showing that the left hand side is a measurable function of $x$.

To complete the proof of the first assertion, we can approximate every bounded
measurable function $F$ on $I\times J$ by stepfunctions $G_n$ uniformly, and
then $F_x$ is also approximated by the corresponding stepfunctions $(G_n)_x$
uniformly. By Lemma \ref{LEM:HAT-CONT}(c), the $\overline{\AA}$-measurable
functions $\wh\fg((G_n)_x)$ approximate $\wh\fg(F_x)$ uniformly, proving that
$\wh\fg(F_x)$ is $\overline{\AA}$-measurable.

\medskip

Turning to the proof of the second assertion, we need the following fact.

\medskip

\noindent{\bf Claim 2.} {\it There is a sequence of measurable functions
$f_n:~J\to[0,1]$ such that $f_n\to g$ $\pi$-almost everywhere, and
\begin{equation}\label{EQ:INEQ-N}
\limsup_{n\to\infty} \wh\fg(f_n) \le \intl_I \wh\fg(F_x)\,d\lambda(x).
\end{equation}}

Let $\xb=(x_0,x_1,\ldots)$ be an infinite sequence of points in $I$, and let
\[
f_n(\xb,y) =\frac1n\sum_{i=0}^{n-1} F(x_i,y).
\]
Then
\begin{equation}\label{EQ:FIN-INEQ}
\wh\fg(f_n(\xb,.)) \le \frac1n\sum_{i=0}^{n-1}\wh\fg(F_{x_i})
\end{equation}
by Theorem \ref{THM:SUBMOD-UNCROSS2}. Let $\XX$ be the set of pairs $(\xb,y)$
such that $\xb=(x_0,x_1,\ldots)\in I^\Nbb$, $y\in J$, and $f_n(\xb,y)\to g(y)$.
For every $y\in J$, this happens for $\lambda^\Nbb$-almost all sequences $\xb$,
so $\lambda^\Nbb\times\pi(\XX)=1$. Hence $f_n(\xb,y)\to g(y)$ holds for
$\pi$-almost all $y$ and $\lambda^\Nbb$-almost all $\xb$.

By the Law of Large Numbers,
\begin{equation}\label{EQ:WHPHI-CONV}
\frac1n\sum_{i=0}^{n-1}\wh\fg(F_{x_i})\to \intl_I \wh\fg(F_x)\,d\lambda(x)\qquad(n\to\infty)
\end{equation}
for $\lambda^\Nbb$-almost all $\xb$. So we can fix a sequence $\xb$ such that
the functions $f_n=f_n(\xb,.)$ are measurable, $f_n(y)\to g(y)$ for
$\pi$-almost all $y$, and \eqref{EQ:INEQ-N} holds. This proves the Claim.

\medskip

Fix a sequence $\xb$ and the functions $f_n$ as in the Claim. Let $h_n=g-f_n$,
then $-1\le h_n\le 1$, and by Theorem \ref{THM:SUBMOD-UNCROSS2},
$\wh\fg(g)=\wh\fg(f_n+h_n)\le\wh\fg(f_n)+\wh\fg(h_n)$. Notice that $h_n\to0$
$\pi$-almost everywhere, which implies that $\pi\{h_n\ge t\}\to 0$ for every
$t>0$. Since $\fg$ is uniformly continuous with respect to $\pi$, it follows
that $\fg\{h_n\ge t\}\to 0$ for every $t>0$. Hence by dominated convergence,
\begin{equation}\label{EQ:DOMCONV1}
\intl_0^1 \fg\{h_n\ge t\}\,dt\to 0.
\end{equation}
Similarly, $\pi\{h_n< t\}\to 0$ for every $t<0$, which implies that
$\fg\{h_n\ge t\}\to \fg(J)$. Hence
\begin{equation}\label{EQ:DOMCONV1}
\intl_{-1}^0 \fg\{h_n\ge t\}\,dt = \intl_{-1}^0 \fg\{h_n\ge t\}\,dt \to \fg(J).
\end{equation}
So
\[
\wh\fg(h_n) = \intl_{-1}^1 \fg\{h_n\ge t\}\,dt -\fg(J)
= \intl_{-1}^0 \fg\{h_n\ge t\}\,dt + \intl_0^1 \fg\{h_n\ge t\}\,dt-\fg(J)\to 0.
\]
and therefore
\[
\wh\fg(g) \le \liminf_n \wh\fg(f_n).
\]
Combined with \eqref{EQ:INEQ-N}, this proves the theorem.
\end{proof}

The continuity condition cannot be omitted, even if $\fg$ is modular, as shown
by the following example.

\begin{example}\label{EXA:NOT-FUBINI0}
Let $I$ be the interval $(0,1)$, and let $\lambda$ be the uniform measure on
$I$. For $x\in I$, let $x_i$ denote the $i$-th bit of $x$ (where tailing
all-$1$ sequences are excluded). Let $J=\{ij:~i,j\in\Nbb,~i<j\}$, $\BB=2^J$,
and let $F(x,ij)=\one(x_i=1,x_j=0)$ ($ij\in J,\,x\in I$). Clearly $F$ is
measurable.

For $x\in I$, let $A_x=\{ij\in J:~F(x,ij)=1\}$. It is easy to see that
$\cup_{x\in I} A_x = J$. On the other hand, the union of any finite family of
sets $A_x$ is a proper subset of $J$. Indeed, consider any finite set
$S\subseteq (0,1)$. There must be two integers $i<j\in\Nbb$ such that $x_i=x_j$
for every $x\in S$. But then $ij\notin A_x$ for any $x\in S$.

Let $\Cf$ be the family of all finite unions of sets $A_x$ and their subsets,
then $\Cf\subseteq 2^J$ is an ideal, which does not contain $J$. So it can be
extended to a maximal ideal $\Df$. Since $\cup \Df=J$, $\Df$ is not principal.
Let $\alpha(X)=\one(X\notin\Df)$, then $\alpha$ is a charge on $(J,2^J)$.

We have
\[
g(ij)=\intl_I F(x,ij)\,d\lambda(x) =\lambda\{x:~x_i=1,x_j=0\} = \frac14,
\]
so $\wh\alpha(g) = \frac14$. On the other hand, $\alpha(A_x)=0$, and hence
\[
\intl_I \wh\alpha(\one_{A_x})\,d\lambda(x) = \intl_I \alpha(A_x)\,d\lambda(x) = 0.
\]
So even the lopsided version of Fubini's Theorem fails for this example.
\end{example}

\paragraph{Acknowledgement.} The author is indebted to the referee for
suggesting a number of corrections, and for encouraging to formulate the
results in the more general setting of set-algebras rather than sigma-algebras.
I am also grateful to the informal research group at the Rényi Institute for
their enthusiastic participation in the research of submodularity bridging the
combinatorial and measure-theoretic lines, and thereby also inducing me to
write this paper.

\end{document}